\documentclass[11pt, a4paper]{article}
\usepackage[utf8]{inputenc}
\usepackage{amsmath, amsthm, amsfonts, amssymb}
\usepackage[bookmarks]{hyperref}

\newtheorem{lemma}{Lemma}[section]

\newtheorem{corollary}[lemma]{Corollary}
\newtheorem*{mainth}{Main Theorem}
\theoremstyle{definition}

\newcommand{\abs}[1]{\ensuremath{\left| #1 \right|}}
\newcommand{\op}{\operatorname}
\newcommand{\ce}[2]{\op{C}_{#1}(#2)}

\newcommand{\ze}[1]{\op{\mathcal{Z}}(#1)}
\newcommand{\fra}[1]{\op{\Phi}(#1)}

\newcommand{\rad}[2]{\op{O}_{#1}(#2)}
\newcommand{\syl}[2]{\op{Syl}_{#1}\left(#2\right)}
\newcommand{\hall}[2]{\op{Hall}_{#1}\left(#2\right)}
\newcommand{\aut}[1]{\op{\textup{Aut}}({#1})}
\newcommand{\out}[1]{\op{\textup{Out}}({#1})}

\title{On products of groups and indices not divisible by a given prime}

\author{Mar\'{\i}a Jos\'e Felipe\footnote{Instituto Universitario de Matem\'atica Pura y Aplicada (IUMPA), Universitat Polit\`ecnica de Val\`encia, Camino de Vera, s/n, 46022, Valencia, Spain.}, Lev S. Kazarin\footnote{Department of Mathematics, Yaroslavl P. Demidov State University,  Sovetskaya Str 14, 150014 Yaroslavl, Russia.\newline 
mfelipe@mat.upv.es, Kazarin@uniyar.ac.ru,  anamarti@mat.upv.es, vicorso@doctor.upv.es},\\ Ana Mart\'inez-Pastor\footnotemark[1], and V\'ictor Sotomayor\footnotemark[1]} 

\date{\textit{Dedicated to the memory of Carlo Casolo}}

\begin{document}

\maketitle
\begin{abstract}
\noindent Let the group $G = AB$ be the product of subgroups $A$ and $B$, and let $p$ be a prime. We prove that  $p$ does not divide the conjugacy class size (index) of each $p$-regular element of prime power order $x\in A\cup B$ if and only if $G$ is $p$-decomposable, i.e. $G=O_p(G) \times O_{p'}(G)$. 
\end{abstract}

\section{Introduction and statement of results}

All groups considered in this paper are finite. In recent years a new research line its being  developed in the confluence of two well-established areas of study in group theory. On the one hand the theory of products of groups, and on the other hand the study of the influence of conjugacy class sizes, also called  \emph{indices}, on the structure of finite groups.  The present paper is a  contribution to this current development. 
\smallskip

Regarding products of groups, the main objective is to obtain information about the structure of a factorised group from the one of the subgroups in the factorization (and vice versa). In this setting, the fact that a product of two normal supersoluble groups is not necessarily supersoluble, has led  to the approach of assuming certain permutability relations between the factors involved (see \cite{BEA} for a detailed account). In particular, among others, \emph{mutually permutable products} have been considered. These are factorised groups $G = AB$ such that $A $ permutes with every subgroup of $B$ and $B$ permutes with every subgroup of $A$. %Clearly this property is satisfied if $A$ and $B$ are normal subgroups in $G$.

\smallskip
Besides, during the last decades, several authors have carried out in-depth investigations with the purpose of understanding how the structure of a finite group is affected by the indices of its elements. In particular, it has been examined whether the indices of some subsets of elements are enough in order to provide features of the group. The survey \cite{CC} gives a general overview about this subject until 2011. 

\smallskip
In the recent approach, which combines the above mentioned lines within the theory of groups, the main aim is to analyze how the  indices of some elements in the factors of a factorised group influence the structure of the whole group. Most of the contributions in this framework consider additional hypotheses on the subgroups in the factorization. Some of them (\cite{LWW, ZGS}) impose some (sub)normality conditions on either both factors. Other papers consider mutually permutable products (see \cite{BCL, CL, FMOsquare}). Recent work done by some of the authors (\cite{FMOvan, FMOpi}) extends previous developements by considering some special type of factorisations, the so-called \emph{core-factorisations}. Only \cite{FMOprime} treats  prime power indices without considering any additional restriction on the factors. On the other hand, it is to be said that in most cases the conditions on the indices are imposed only on some subsets of elements of the factors, namely prime power order elements, $p$-regular elements, zeros of irreducible characters, etc.

% the conjugacy class of $x$ in $G$ is denoted by $x^G$, and its size is the \emph{index of $x$} denoted 
\smallskip
The notation and terminology are as follows. For an element $x$ in a group $G$, we call  $i_G(x)$ the \emph{index of $x$}, i.e $i_G(x)=|G:\ce{G}{x}|$. A \emph{$p$-regular} element is an element whose order is not divisible by $p$, where $p$ will always be a prime number. If $n$ is a positive integer, $n_p$ denotes the highest power of $p$ dividing $n$. We represent by  $\pi(G)$ the set of all prime divisors of $|G|$. The set of all Sylow $p$-subgroups of $G$ is $\syl{p}{G}$ and $\hall{\pi}{G}$ is the set of all Hall $\pi$-subgroups of $G$ for a set of primes $\pi$. A group such that $G=\rad{\pi}{G}\times \rad{\pi'}{G}$ is said to be \emph{$\pi$-decomposable}. If $H$ is a subgroup of $G$, we denote by $H^G$ the normal closure of $H$ in $G$. The remaining notation and terminology are standard within the theory of finite groups, and they mainly follow those of the book \cite{DH}, apart from some terminology on simple groups which will be introduced later. 

%the set of all prime divisors of a natural number $n$ by $\pi(n)$ 
\smallskip
It is well known that  if $p$ does not divide  $i_G(x) $ for every $p$-regular element in a group $G$, then the Sylow $p$-subgroup is a direct factor of $G$ (see for instance \cite[Lemma 2]{CC}). This result was improved in \cite[Theorem 5]{LWW} by proving that the same conclusion remains true if the conditions on the indices are only imposed on $p$-regular elements of prime power order.  
In this paper, we deal with the corresponding result for factorised groups, but avoiding the consideration of any additional conditions on the factors, as were considered in \cite{BCL, FMOpi, ZGS}. This means that, in contrast to some of the mentioned results whose proofs are elementary,  the classification of finite simple groups (CFSG) has been used in our proof. In particular, we derive some results on the center of the prime graph of an almost simple group, which will be used as a tool. 
\smallskip

The aim of this paper is then to prove the following result:
\begin{mainth}\label{mainth} Let the group $G = AB$ be the product of subgroups $A$ and $B$, and let $p$ be a prime. Then $p$ does not divide $i_G(x)$ for every $p$-regular element of prime power order $x\in A\cup B$ if and only if $G$ is $p$-decomposable, i.e. $G=O_p(G) \times O_{p'}(G)$.

\end{mainth}

Note that if $G$ is $p$-decomposable, then clearly the conditions on the indices hold.  For the converse, the following lemma shows that only the existence of a unique Sylow $p$-subgroup should be proved.

%Note that if $G$ is $p$-decomposable, then clearly the conditions on the indices hold.  For the converse, only the existence of a unique Sylow $p$-subgroup should be proved (see Lemma \ref{pclos}).

\begin{lemma}\label{pclos}
Let the group $G = AB$ be the product of the subgroups $A$ and $B$, and let $p$ be a prime. If $p$ does not divide  $i_G(x)$ for every $p$-regular element of prime power order $x\in A\cup B$, then the following statements are equivalent:
\begin{itemize}
\item[i)] $G$ is $p$-closed, i.e. $G$ has a normal Sylow $p$-subgroup.
\item[ii)] $G$ is $p$-decomposable.
\end{itemize}
\end{lemma}
\begin{proof}
Clearly, it is enough to prove that (i) implies (ii). Let $P\in\syl{p}{G}$  and assume that $P \unlhd G$. Since $p$ does not divide  $i_G(x)=|G : \ce{G}{x}|$ it follows that $P\leq  \ce{G}{x}$ for every $p$-regular element of prime power order $x\in A\cup B$. Since $G$ is $p$-separable, by Lemma~\ref{1.3.2}, we may consider $H$ a Hall $p'$-subgroup of $G$ such that $H= (H \cap A)(H \cap B)$. Hence, for every element $x\in (H \cap A) \cup (H \cap B)$ of prime power order, it holds that $P\leq  \ce{G}{x}$. Therefore  $[P, H]=1$ and (ii) follows. 
\end{proof}

\smallskip
As an inmediate consequence of the Main Theorem, we get: 
\begin{corollary}\label{all} Let the group $G = AB$ be the product of subgroups $A$ and $B$, and let $p$ be a prime. Then $p$ does not divide $i_G(x) $ for every element of prime power order $x\in A\cup B$ if and only if  $G$ has a central Sylow $p$-subgroup, i.e. $G=O_p(G) \times O_{p'}(G)$ with $O_p(G)$ abelian.
\end{corollary}

Our results provide an improvement of \cite[Theorem 1.1]{BCL} in the case of only two factors, since in that paper  products of $n$ pairwise mutually permutable subgroups were considered.

\begin{corollary}\textup{\cite[Theorem 1.1]{BCL}}\label{adolfo} Let the group $G = AB$ be the mutually product of the subgroups $A$ and $B$, and let $p$ be a prime. Then:
\begin{itemize}
\item[i)] No index $i_G(x) $, where $x$ is a $p$-regular element in  $ A\cup B$, is divisible by $p$ if and only if  $G=O_p(G) \times O_{p'}(G)$.
\item[ii)] $i_G(x) $ is not divisible by $p$ for every element $ x \in A\cup B$ if and only if  $G=O_p(G) \times O_{p'}(G)$ with $O_p(G)$ abelian.
\end{itemize}

\end{corollary}

Finally, we also point out that \cite[Theorem A]{FMOpi} and \cite[Theorem 3.2]{ZGS} when $\pi=p'$ are direct consequences from our main result.

\section{Preliminary results} 

We will use without further reference the following elementary lemma:
\begin{lemma}
Let $N$ be a normal subgroup of a group $G$, and $H$ be a subgroup of $G$. We have:
\begin{itemize}
	\item[\emph{(a)}] $i_N(x)$ divides $i_G(x) $, for each $x\in N$.
	
	\item[\emph{(b)}] $i_{G/N}(xN)$ divides $i_G(x) $, for each $x\in G$.
	
	\item[\emph{(c)}] If $xN$ is a $\pi$-element of $HN/N$, for a set of primes $\pi$, then there exists a $\pi$-element $x_1\in H$ such that $xN = x_1N$.
\end{itemize}
\end{lemma}

We will also need the following fact about Hall subgroups of factorised groups, which is a convenient reformulation of \cite[1.3.2]{AFG}. We recall that a group is  a D$_{\pi}$-group, for a set of primes $\pi$, if every $\pi$-subgroup is contained in a Hall $\pi$-subgroup, and any two Hall $\pi$-subgroups are conjugate. It is well known that any $\pi$-separable group is a  D$_{\pi}$-group. Also, all finite groups are D$_{\pi}$-groups when $\pi$ consists of a single prime.

\begin{lemma}\label{1.3.2}
Let $G=AB$ be the product of the subgroups $A$ and $B$. Asume that $A, B$, and $G$ are D$_{\pi}$-groups for a set  of primes $\pi$. Then there exists a Hall $\pi$-subgroup $H$ of $G$ such that $H= (H \cap A)(H \cap B)$, with $H \cap A$ a Hall $\pi$-subgroup of $A$ and $H \cap B$ a Hall $\pi$-subgroup of $B$.
\end{lemma}

Next we record some arithmetical lemmas, that will be applied later.

\begin{lemma}\textup{(\cite[Lemma 6]{KMP3})}\label{SylowSym}
Let $G$ be the symmetric group of degree $k$ and let $s$ be a prime. If $s^{ N}$ is the largest power of $s$ dividing $|G|=k!$, then $N \leq \frac{k-1}{s-1}$.
\end{lemma}

%\begin{lemma}\textup{(\cite[Lemma 6(iii)]{Zav})}\label{cuentas}
%\begin{enumerate}
%\item[\emph{(i)}] If $c>1$ and $s\geq 1$ are natural numbers, then $c^2+c+1 \mid c^{2s}+c^s+1$ if and only if $3\nmid s$.
%\item[\emph{(ii)}] Let $s\in\mathbb{N}$ satisfy $s\equiv 1 (\operatorname{mod} r)$, where $r$ is an odd prime. If $r^l \mid\mid (s-1)$, then $r^{l+m}\mid\mid(s^{r^{m}}-1)$ for each $m\geq 0$.
%\item[\emph{(iii)}] Let $a,s,t$ be natural numbers. Then
%\begin{enumerate}
%\item[\emph{(a)}] $(a^s-1,a^t-1)=a^{(s,t)}-1$,
%\item[\emph{(b)}] $(a^s+1,a^t+1)=\begin{cases} a^{(s,t)}+1 \quad \text{if both } s/(s,t) \text{ and } t/(s,t) \text{ are odd,}\\
%(2,a+1) \quad \text{otherwise,}\end{cases}$
%\item[\emph{(c)}] $(a^s-1,a^t+1)=\begin{cases} a^{(s,t)}+1 \quad \text{if } s/(s,t) \text{ is even and } t/(s,t) \text{ is odd,}\\
%(2,a+1) \quad \text{otherwise.}\end{cases}$
%\end{enumerate}
%\end{enumerate}
%\end{lemma}

\begin{lemma}\textup{(\cite[Lemma 6(iii)]{Zav})}\label{cuentas}
 Let $q,s,t$ be positive integers. Then:
\begin{enumerate}
\item[\emph{(a)}] $(q^s-1,q^t-1)=q^{(s,t)}-1$,
\item[\emph{(b)}] $(q^s+1,q^t+1)=\begin{cases} q^{(s,t)}+1 \quad \text{if both } s/(s,t) \text{ and } t/(s,t) \text{ are odd,}\\
(2,q+1) \quad \text{otherwise,}\end{cases}$
\item[\emph{(c)}] $(q^s-1,q^t+1)=\begin{cases} q^{(s,t)}+1 \quad \text{if } s/(s,t) \text{ is even and } t/(s,t) \text{ is odd,}\\
(2,q+1) \quad \text{otherwise.}\end{cases}$
\end{enumerate}

\end{lemma}

We introduce now some additional terminology. Let $n$ be a positive integer and $p$ be a prime number. A prime $r$ is said
to be \textit{primitive with respect to the pair $(p, n)$} (or a \textit{primitive prime divisor of $p^n-1$}) if
$r$ divides $p^n-1$ but $r$ does not divide $p^k-1$ for every
integer $k$ such that $1\leq k< n$.

\begin{lemma}[Zsigmondy, \cite{Zsi}]\label{Zsi}
Let $n$ be a positive integer and $p$ a prime. Then:
\begin{itemize}
\item[\emph{(a)}] If $n \geq 2$, then there exists a prime $r$ primitive with respect to the pair $(p, n)$ unless $n=2$ and $p$ is a Mersenne prime or $(p, n)=(2, 6)$.
\item[\emph{(b)}] If the prime $r$ is primitive with respect to the pair $(p, n)$, then $r-1 \equiv 0\,(\mbox{mod }n)$. In particular, $r \geq n+1$.
\end{itemize}

\end{lemma}

The following lemmas are used when dealing with prime power order elements. We remark that the proof of the first one uses CFSG. 

\begin{lemma}\textup{(\cite[Theorem 1]{FKS})}\label{FKS}
Let $G$ be a group acting transitively on a set $\Omega$ with $|\Omega|>1$. Then there exists a prime power order element $x\in G$ which acts fixed-point-freely on $\Omega$.
\end{lemma}

\begin{lemma}\label{feinkantor}
Let $H$ be a subgroup of a group $G$. If every prime power order element of $G$ lies in $\bigcup_{g\in G} H^g$, then $G=H$.
\end{lemma}
\begin{proof}
If $H$ is normal in $G$,  then every prime power order element belongs to $H$, and since $G$ is generated by such elements, we get $G=H$. So we may assume that $H$ is not normal in $G$. Note that $G$ acts on $\Omega:=\{H^g\: : \: g\in G\}$ transitively. If $H<G$, then certainly $|\Omega|>1$ and, by Lemma \ref{FKS}, there exists a prime power order element $x\in G$ acting fixed-point-freely on $\Omega$. But the hypotheses imply that $x\in H^z$ for some $z\in G$, so $H^{zx}=H^z$ and this is a contradiction.
\end{proof}

%We will use also the next lemma which provides an equivalent statement to our Main Theorem.
%
%\begin{lemma}\label{pclos}
%Let the group $G = AB$ be the product of the subgroups $A$ and $B$, and let $p$ be a prime. If $p$ does not divide  $\abs{x^G}$ for every $p$-regular element of prime power order $x\in A\cup B$, then the following statements are equivalent:
%\begin{itemize}
%\item[i)] $G$ is $p$-closed, i.e. $G$ has a normal Sylow $p$-subgroup.
%\item[ii)] $G$ is $p$-decomposable.
%\end{itemize}
%\end{lemma}
%\begin{proof}
%Clearly, it is enough to prove that (i) implies (ii). Let $P\in\syl{p}{G}$  and assume that $P \unlhd G$. Since $p$ does not divide  $\abs{x^G}=|G : \ce{G}{x}|$ it follows that $P\leq  \ce{G}{x}$ for every $p$-regular element of prime power order $x\in A\cup B$. Since $G$ is $p$-separable, by Lemma~\ref{1.3.2}, we may consider $H$ a Hall $p'$-subgroup of $G$ such that $H= (H \cap A)(H \cap B)$. Hence, for every element $x\in (H \cap A) \cup (H \cap B)$ of prime power order, it holds that $P\leq  \ce{G}{x}$. Therefore  $[P, H]=1$ and (ii) follows. 
%\end{proof}

\section{Preliminaries on (almost) simple groups and their prime graphs}

We begin this section with a useful result on the centralisers of automorphisms of simple groups, which is a refinement of  \cite[Lemma 2.6]{DPSS}. In fact, the own proof 
of that lemma provides this stronger result: 
\begin{lemma}\label{aut}
Let $N$ be a simple group. Then there exists $r\in \pi(N) \setminus\pi(\op{Out}(N))$ such that $(r, \abs{\ce{N}{\alpha}})=1$ for every non-trivial $\alpha \in \op{Out}(N)$ of order coprime to $\abs{N}$.
\end{lemma}
\begin{proof}
Following the proof of \cite[Lemma 2.6]{DPSS}, we can assume that $N=G(q)$ is a simple group of Lie type, with $q=p^{e}$, $p$ a prime and $e \geq 3$ a positive integer. In that proof it is shown that the prime $r$ is in fact a primitive prime divisor of $p^{em}-1$ for some integer $m \geq 2$, and that such $r$ always exist under the given assumptions. Now having in mind the orders of the outer automorphisms of the simple groups of Lie type (see for instance \cite[Table 2.1]{LPS}) and applying Lemma \ref{Zsi}  we can deduce that $r \not\in \pi(\op{Out}(N))$ (see also  \cite[2.4. Proposition B]{LPS}). 
\end{proof}

We will denote the \emph{prime graph} of a group $G$, also called, the \emph{Gr\"unberg-Kegel} graph, by $\Gamma(G)$. The set of vertices of such graph is the set $\pi(G)$ of prime divisors of $|G|$, and two vertices $r,s$ are adjacent in  $\Gamma(G)$ if there exists an element of order $rs$ in $G$.  The connected components of the prime graph of a simple group are known from \cite{Wil} and \cite{Kon}. We will denote by $\mathcal{Z}( \Gamma(G))$, the center of the graph, i.e. $\mathcal{Z}( \Gamma(G))=\{p \,  | \,  p \mbox{ is adjacent to } r, \forall r \in \pi(G)\}$.

The following result on the center of the prime graph of alternating and symmetric groups will be used later:

\begin{lemma}\label{angraph} Let $n \geq 5, n\neq 6$ be a positive integer. Let $k$ be the largest positive integer such that $\{n, n-1, \ldots, n-k+1\}$ are consecutive composite numbers.  If $k=1$, then both $\Gamma(A_n)$ and $\Gamma(\Sigma_n)$ are non-connected. For $k \geq 2$, let $t$ be the largest prime number such that $t \leq k$. Then:
\begin{itemize}
\item If $k=2$, then $\Gamma(A_n)$ is non-connected and $\ze{\Gamma(\Sigma_n)}=\{2\}$.
\item If $k=3$, then $\ze{\Gamma(A_n)}=\{3\}$ and $\ze{\Gamma(\Sigma_n)}=\{2,3\}$.
\item If $k\geq 4$, then $\ze{\Gamma(A_n)}=\ze{\Gamma(\Sigma_n)}=\{s\in\pi(\Sigma_n)\, | \, 2\leq s \leq t\}$.
\end{itemize}
\end{lemma}
\begin{proof} It is well known that two odd primes $s, u$ are adjacent in $\Gamma(A_n)$, and so in $\Gamma(\Sigma_n)$, if and only if $s+u \leq n$. On the other hand, if $s$ is an odd prime, $s$ is adjacent to $2$ in $\Gamma(A_n)$ if and only if $s+4 \leq n$, and $2,s$ are adjacent in $\Gamma(\Sigma_n)$ only when $s+2 \leq n$. It is then clear that if $k=1$, then both $\Gamma(A_n)$ and $\Gamma(\Sigma_n)$ are non-connected.

Consider the prime $r:=n-k$, which is the largest prime divisor of  $n!$ by the choice of $k$. Clearly,  $r>\frac{n}{2}>k=n-r \geq t$. Thus $r+t \leq n$, and we deduce that  $t\in \op{Z}(\Gamma(\Sigma_n))$.

If $k=2$, then $r=n-2$, and so $\ze{\Gamma(\Sigma_n)}=\{2\}$. Since $r+4>n$, then $\Gamma(A_n)$ is non-connected.

If $k=3$, then $r=n-3$. It follows that  $\ze{\Gamma(\Sigma_n)}=\{2, 3\}$ and $\ze{\Gamma(A_n)}=\{3\}$.

Finally, let us suppose that $k\geq 4$, so $n\geq 11$. Take a prime $s\in\pi(\Sigma_n)$. If $s\leq t$, then $r+s\leq r+t\leq r+k=n$, and so $s$ lies in $\ze{\Gamma(\Sigma_n)}$. Assume now $s>t$, so $s>k$. It is known that  there exist two primes $\frac{n}{2}<r_1<r_2\leq n$, and we may take $r_2=r$. If $s\neq r$, then $s+r=s+n-k>n$. If $s=r$, then $s+r_1>\frac{n}{2}+\frac{n}{2}=n$. Hence, in both cases $s\notin \ze{\Gamma(\Sigma_n)}$. This proves that $\ze{\Gamma(\Sigma_n)}=\{s\in\pi(\Sigma_n)\, : \, 2\leq s \leq t\}$. Finally, since $k\geq 4$, then $r\leq n-4$, so $r+4\leq n$ and $2\in \op{Z}(\Gamma(A_n))$. Therefore $\ze{\Gamma(A_n)}=\ze{\Gamma(\Sigma_n)}$.
\end{proof}

For the special case of the alternating group $A_6$ and its group of automorphisms we can derive the following result from \cite[Lemma 2]{Luc}:
\begin{lemma}\label{a6}
If $A_6=N\unlhd G \leq  \op{Aut}(N)$, then $\Gamma(G)$ is  non-connected, except when $G=\op{Aut}(N)$. In this last case $\ze{\Gamma(G)}=\{2\}$. 
\end{lemma}

Also, for sporadic groups the following result is well known (see \cite{Atl} or \cite[Theorem 2]{Wil} and \cite[Theorem 3]{Luc}):

\begin{lemma}\label{sporgraph}
If $N$ is an sporadic simple group, then $\Gamma(N)$ is non-connected. Moreover, $\Gamma(\aut{N})$ is also non-connected, except when $N= McL$ or $N=J_2$, and in both cases $\ze{\Gamma(\aut{N})}=\{2\}$.
\end{lemma}

We will use the following facts on groups of Lie type in the proof of our Main Theorem.
In the sequel, for $q=t^e$, $e\ge 1$, we will denote by $q_n$ \emph{any} primitive  prime divisor of $t^{en}-1$, i.e. primitive with respect to  $(t, ne)$.

\begin{lemma}\label{class}

For $N=G(q)$  a classical simple group of Lie type of characteristic $t$  and  $q=t^e$,   there exist primes $r, \, s \in \pi(N) \setminus (\pi(\out{N}) \cup \{t\})$ and maximal tori $T_1$ and $T_2$ of $N$,  of orders divisible by $r$ and $s$, respectively, with $( |T_1|, |T_2|)=1$,  as stated in Table 1. (In such table for the case $\star$, $l$   denotes a Mersenne prime.)

Moreover, for the groups $N$ listed in Table 2, there exist a prime $s \in \pi(N) \setminus (\pi(\out{N}) \cup \{t\})$ and a Sylow $s$-subgroup of order $s$ which is self-centralising in $N$. 

\smallskip
If $N=L_2(q)$, $C_N(x)$ is a $t$-group for each $t$-element $x \in N$.

\smallskip
If $N=L_3(q)$, there exists a maximal torus $T$  of order $(1/d) (q^2+q+1)$, $d=(3, q-1)$, such that each prime $r \in \pi(T)$ is a primitive prime divisor of $q^3-1$ (for $q\neq 4$), and $(|T|, 2t)=1$.

\smallskip
If $N=U_3(q)$, there exists a maximal torus $T$  of order $(1/d) (q^2-q+1)$, $d=(3, q+1)$,  such that each prime $r \in \pi(T)$ is a primitive prime divisor of $q^6-1$, and $(|T|, 2t)=1$.
\begin{table}[ht!]
$$\begin{array}{c|c|c|c|c|c}
\hline
& & & & & \\
 N& r &s& |T_1|& |T_2| & Remarks \\
& & & & & \\
\hline
& & & & & \\
 L_n(q)&  q_n &q_{n-1} & \frac{q^n-1}{(n, q-1)(q-1)}&\frac{q^{n-1}-1}{(n, q-1)}& (n, q) \neq (6, 2)\\
n \geq 4 & & & & & (n, q) \neq (4, 4), (7, 2)\\
&& s=7&&&(n,q)=(4,4)\\
& & & & &  \\
\hline
& & & & & \\
U_{n}(q) &  q_{n} & q_{2(n-1)} &\frac{q^{n}-1}{(n, q+1)(q+1)}& \frac{(q^{n-1}+1)}{(n, q+1)}&  n \mbox{ even }  \\
 & & & & & (n, q) \neq (4, 2), (6, 2) \\
  & & & & & \\
   n\geq 4 &  q_{2n} & q_{n-1} & \frac{q^{n}+1}{(n, q+1)(q+1)} & \frac{q^{n-1}-1}{(n, q+1)}& n  \mbox{ odd} \\
 & & & & & (n, q) \neq (7, 2)  \\
& & & & &  \\
\hline
& & & & & \\
 PSp_{4}(q) &  q_{4} & q_{2} & \frac{q^{2}+1}{(2, q-1)}& \frac{(q^{2}-1)}{(2, q-1)}& q\neq 8, l  \quad (\star) \\
&& s=7&&& q=8\\
&& s \neq 2&&& q=l \\
& & & & &  \\
\hline
& & & & & \\
 PSp_{2n}(q) &  q_{2n} & q_{2(n-1)} & \frac{q^{n}+1}{(2, q-1)}& \frac{(q^{n-1}+1)(q-1)}{(2, q-1)}&   n \mbox{ even }  \\
 & && & & (n, q) \neq (4,2) \\
  P\Omega_{2n+1}(q) & & & & & \\
   &  q_{2n} & q_{n} & \frac{q^{n}+1}{(2, q-1)} &  \frac{(q^{n}-1)}{(2, q-1)}& n  \mbox{ odd }    \\
  n\geq 3 &  & & & & (n, q) \neq (3, 2)  \\
& & & & &  \\
\hline
& & & & & \\
P\Omega_{2n}^{-}(q)  & q_{2n} & q_{2(n-1)} & \frac{q^{n}+1}{(4, q^n+1)}& \frac{(q^{n-1}+1)(q-1)}{(4, q^n+1)}&  (n, q) \neq (4, 2) \\
n \geq 4 &&&&&\\
& & & & &  \\

\hline
& & & & &  \\
P\Omega_{2n}^{+}(q) & q_{2(n-1)} & q_{n-1} &\frac{(q^{n-1}+1)(q+1)}{(4, q^n-1)}  &\frac{(q^{n-1}-1)(q-1)}{(4, q^n-1)} &  n \mbox{ even }  \\
 & & &&&(n,q)\neq (4,2)\\

  n \geq 4  & q_{2(n-1)} &  q_{n} &\frac{(q^{n-1}+1)(q+1)}{(4, q^n-1)} & \frac{q^n-1}{(4, q^n-1)} & n \mbox{ odd}\\
  & & & & &  \\
 \hline
 \end{array}$$
 \caption{Maximal tori for classical groups}\label{tori}
 \end{table}

\begin{table}[h]
\centering
\[ \begin{array}{|c|c|} \hline
\, N  \, & \, s \, \\
 \hline
\, L_3(4)  \, & \quad 7 \quad \\
 \hline
\, L_6(2)  \, & \quad 31 \quad \\

 \hline
\, L_7(2)  \, & \quad 127 \quad \\

 \hline
\, U_6(2)  \, & \quad 11 \quad \\

 \hline
\, U_7(2)  \, & \quad 43 \quad \\

\hline
\, PSp_{4}(4)  \, & \quad 17 \quad \\

 \hline
\, PSp_{6}(2)  \, & \quad 7 \quad \\

\hline
\, PSp_{8}(2)  \, & \quad 17 \quad \\

 \hline
\, P\Omega_{8}^{-}(2)  \, & \quad 17 \quad \\

 \hline
\, P\Omega_{8}^{+}(2)  \, & \quad 7 \quad \\

 \hline
\end{array}
\]\vspace{-.1cm}
\caption{Self-centralising Sylow $s$-subgroups of order $s$}\label{self}
\end{table}

\end{lemma}
\begin{proof} We recall that a torus is an abelian $t'$-group. The existence of the subgroups $T_1$ and $T_2$ appearing in Table 1 can be derived from the known facts about the maximal tori in these groups (see, \cite[Propositions 7-10]{Car} or \cite[Lemma 1.2]{VV}). The fact that the corresponding orders of the tori are coprime in each case can be deduced easily from Lemma \ref{cuentas}, while the assertion regarding the primitive prime divisors is deduced from Lemma \ref{Zsi}. 

The information in Table 2 can be found either in \cite{Atl}, or from the orders of maximal tori for the corresponding groups.

Note that the case $PSp_{4}(2)\cong \Sigma_6$ has already been considered in Lemma \ref{a6}. 

The assertion on $L_2(q)$ is well known (see for instance  \cite[Proposition 7]{Car}).

The existence of tori of the corresponding orders in $L_3(q)$ and $U_3(q)$ can be found in  \cite[Lemma 1.2]{VV}, and the claim on the prime divisors is easily deduced applying Lemma \ref{cuentas}.
\end{proof}

\newpage
%For $N=L_4(4)$ take $r=17$, $s= 7$ and tori of order $85$ and $21$, respectively. 

%For $N \cong L_6(2)$ or $N \cong L_7(2)$, in both cases, if $s$ is the largest prime number dividing $|N|$, then  $s \not \in\pi(G/N)$ and a Sylow $s$-subgroup of $N$ is of order $s$ and self-centralising in $N$.
%
%For $N=U_4(2)$, there exists a Sylow $5$-subgroup of order $5$ self-centralising in $N$.
%
% For  $N=U_6(2)$, there exists a Sylow $11$-subgroup of order $11$ self-centralising in $N$. 
%
%For  $N=U_7(2)$, there exists a Sylow $43$-subgroup of order $43$ self-centralising in $N$.
%
%For $N=PSp_{4}(4)$ and  $N=PSp_{8}(2)$ there exists a Sylow $17$-subgroup of order $17$ self-centralising in $N$.
%
%The case $PSp_{4}(2)\cong \Sigma_6$ has been considered in Lemma \ref{angraph}.
%
%For $N=PSp_{6}(2)$ , there exists a Sylow $7$-subgroup of order $7$ self-centralising in $N$.
%
%
%For $N=P\Omega_{8}^{-}(2)$  there exists a Sylow $17$-subgroup of order $17$ self-centralising in $N$.
%For $N=P\Omega_{8}^{+}(2)$  there exists a Sylow $7$-subgroup of order $7$ self-centralising in $N$.

\begin{lemma}\label{excep}
For $N=G(q)$   an exceptional  simple group of Lie type of characteristic $t$  and  $q=t^e$,   there exist primes $r, \, s \in \pi(N) \setminus (\pi(\out{N}) \cup \{t\})$ and maximal tori $T_1$ and $T_2$ of $N$,  of orders divisible by $r$ and $s$, respectively, with $( |T_1|, |T_2|)=1$,  as stated in Table 3.

In the cases denoted by $(\star)$, $r$ and $s$ denote the largest prime divisor of $|T_1|$ and $|T_2|$, respectively.

The Tits group $N=F_4(2)'$ contains a Sylow 13-subgroup of order 13 which is self-centralising. 

\begin{table}[ht!]
$$\begin{array}{c|c|c|c|c|c}
\hline
& & & & & \\
 N& r &s& |T_1|& |T_2| & Remarks \\
& & & & & \\
\hline
& & & & & \\
G_2(q)&  q_3 &q_{6} & q^2+q+1&q^2-q+1& q\neq 4\\
q > 2 & r=7 & & & & q=4\\
& & & & &  \\
\hline
& & & & & \\
F_{4}(q) &  q_{8} & q_{12} &q^{4}+1 & q^4-q^2+1&  \\
 & & & & &   \\
\hline
& & & & & \\

E_6(q)&  q_{9} & q_{12} & \frac{q^{6}+q^3+1}{(3, q-1)}& \frac{(q^4-q^2+1)(q^2+q+1)}{(3, q-1)}&  \\
& & & & &  \\
\hline
& & & & & \\

E_7(q)&  q_{9} & q_{14} & \frac{(q^6+q^3+1)(q-1)}{(2, q-1)} & \frac{q^{7}+1}{(2, q-1)}&  \\
& & & & &  \\
\hline
& & & & & \\
E_8(q)&  q_{20} & q_{24} & q^8-q^{6}+q^4-q^2+1& q^8-q^4+1&  \\
& & & & &  \\
\hline
& & & & & \\
 {}^3D_4(q)&  q_{3} & q_{12} &(q^{3}-1)(q+1)& q^4-q^2+1&    \\
& & & & &  \\
\hline
& & & & & \\  
{}^2B_2(q)  & r & s & q+\sqrt{2q}+1& q-\sqrt{2q}+1& (\star) \\
q=2^{2m+1} >2 &&&&&\\
& & & & &  \\
\hline
& & & & & \\  
{}^2G_2(q)  & r & s & q+\sqrt{3q}+1& q-\sqrt{3q}+1& (\star) \\
q=3^{2m+1} >3 &&&&&\\
& & & & &  \\
\hline
& & & & & \\  
{}^2F_4(q)  & r & s & \small{q^2+q \sqrt {2q}+q+\sqrt{2q}+1}& (q-\sqrt{2q}+1)(q-1)& (\star) \\
q=2^{2m+1} >2  &&&&&\\
& & & & &  \\

 \hline
& & & & & \\
 {}^{2}E_6(q)&  q_{18} & q_{12} & \frac{q^{6}-q^3+1}{(3, q+1)}& \frac{(q^4-q^2+1)(q^2-q+1)}{(3, q+1)}&  \\

& & & & &  \\
\hline
 \end{array}$$
 \caption{Maximal tori for exceptional groups}\label{tori2}
 \end{table}
\end{lemma}
\begin{proof}
The existence of the subgroups $T_1$ and $T_2$ appearing in Table 3 can be derived from the information about the maximal tori in these groups (see \cite[Lemma 1.3]{VV} and \cite[Lemma 2.6]{VV2}).

The fact that they are coprime can be deduced from Lemma \ref{cuentas} having in mind that $|T_i|$ divides $q^n-1$ when we state that $q_n \in \pi(T_i)$, $i=1, 2$,  (for the case  ${}^3D_4(q)$, $|T_1|$ divides $q^6-1$), for all groups except ${}^2B_2(q), {}^2G_2(q), {}^2F_4(q) $. In the latter cases, denoted by $(\star)$,  the information can be obtained from \cite[Lemma 2.8]{VV2}
\end{proof}
The previous lemmas provide the following result on the center of the prime graph of a simple group of Lie type, which can also be derived from \cite[Proposition 2.9]{center}. 

\begin{corollary}
If $N$ is a simple group of Lie type, then $\ze{\Gamma(N)}=\emptyset$.
\begin{proof}
Let $t$ be the characteristic of  the group of Lie type $N$. It is well known that $t \not \in \ze{\Gamma(N)}$. If $p \in  \ze{\Gamma(N)}$, then for each prime $r \neq t$, there exists an abelian $t'$-subgroup of $N$  whose order is divisible by $p$ and $r$. Since any abelian $t'$-subgroup is contained in a maximal torus, this means that $p \in \pi(T)$ for each maximal torus $T$ of $N$. Hence, the information given in Lemmas \ref{class} and \ref{excep} leads to a contradiction.  
\end{proof}

\end{corollary}

\section{The minimal counterexample: reduction to the almost simple case}

In this section we will give a description of the structure of a minimal counterexample to our Main Theorem. Hence,  having in mind Lemma \ref{pclos}, we
assume the following hypotheses: %RECALL LEmma 1.1
\begin{description}
 \item[(H1)] $p$ is a prime number.

 \item[(H2)] $G$ is a group satisfying the following conditions: 
\begin{enumerate}
\item[(i)] $G=AB$ is the product of the subgroups $A$ and $B$, and  $p$ does not divide $i_G(x) $ for every $p$-regular element  of prime power order $x \in A \cup B$.

\item[(ii)] $G$ does not have a normal Sylow $p$-subgroup.
\end{enumerate}
\end{description}
Among all such groups we choose $(G, A, B)$ such that $|G|+|A|+|B|$ is minimal. \\

For such a group $G$ we have the following results.

\begin{lemma}\label{0}
$G$ has a unique minimal normal subgroup $N$ which is not a $p$-group. Moreover, $P \neq 1$,  $PN \unlhd G$,  $G/N=PN/N \times O_{p'}(G/N)$,  and $G=NN_{G}(P)$, for each $P \in \syl{p}{G}$. 
\end{lemma}
\begin{proof}
Since the hypotheses (H2)(i) are clearly inherited by quotients of $G$, and the class of all $p$-closed groups is a saturated formation, we deduce that $\fra{G}=1$ and that $G$ has a unique minimal normal subgroup, say $N$. Since $G/N$ has a normal Sylow $p$-subgroup, then  $\rad{p}{G}=1$, and so $N$ is not a $p$-group. Also this implies that $PN \unlhd G$  for each $P \in \syl{p}{G}$, and that  $G/N$ is $p$-decomposable, by Lemma~\ref{pclos}, as claimed. The last assertion follows from Frattini's argument.
\end{proof}

From now on $N$ is the unique minimal normal subgroup of $G$.
\medskip

\begin{lemma}\label{1} $G=APN=BPN$, for each $P \in \syl{p}{G}$ .
\end{lemma}
 \begin{proof}
Let $P \in \syl{p}{G}$. Since $PN\unlhd G$, take for instance $T:=APN$. Let us suppose that $T<G$. Note that $T=A(T \cap B)$. If we take any $p$-regular element of prime power order $x\in A\cup (T\cap B)$, since $G=NN_{G}(P)$, then, by our hypotheses, there exists some $n\in N\leq  T$ such that $P^n\leq  \ce{G}{x}$, where $P\in\syl{p}{T}$. Whence $T$ satisfies the hypotheses and, by minimality, we deduce that $P \unlhd T$. But this means, by Lemma~\ref{pclos}, that $T$, and so $N$,  is $p$-decomposable. Since $N$ is not a $p$-group, we deduce that  $N=\rad{p'}{N}\leq \rad{p'}{T}\leq \ce{G}{P}$. But then $P$ is normal in  $G=NN_{G}(P)$, a contradiction. Therefore, $G=APN$ and, analogously, $G=BPN$.
\end{proof}

\begin{lemma}\label{2}
Either $p\in\pi(A)$ or $p\in\pi(B)$. Moreover, if $X\in\{A, B\}$ and $p\in\pi(XN)$, then $G=XN$.
\end{lemma} 
\begin{proof}
The first assertion is clear since $G=AB$ and $p \in \pi(G)$.

Without loss of generality, let us assume that $p\in\pi(AN)$. Consider $1\neq P_0\in\syl{p}{AN}$ and take $P \in \syl{p}{G}$ with $P_0\leq  P$, that is, $P_0=P\cap AN$.

Set $H:=AN$ and observe that $H=A(H \cap B)$. Note that for each $p$-regular element of prime power order $x\in A \cup(H\cap B)$,  there exists some $n\in N$ with $P^n\leq  \ce{G}{x}$. Hence $x\in \ce{H}{P^n}\leq \ce{H}{P_0^{n}}$ with $n\in N$, so $i_H(x)$ is not divisible by $p$. If $H<G$, then, by minimality, we deduce that $1\neq \rad{p}{H}=P_0=P\cap AN\leq  AN=H$. Thus $\rad{p}{H}^G=\rad{p}{H}^P$ because $G=APN$, by Lemma~\ref{1}. It follows that $1\neq \rad{p}{H}^G\leq  P$, and so $\rad{p}{G}\neq 1$, a contradiction.
\end{proof}

\begin{lemma}\label{3} If $p\in\pi(N)$, then $G=AN=BN=AB$ and $N$ is a  non-abelian simple group.  Hence $N \unlhd G \leq \aut{N}$, i.e. $G$ is an almost simple group.

\end{lemma}
 \begin{proof}
The first assertion follows from Lemma~\ref{2}. 
 Since $\rad{p}{G}=1$ and  $p\in\pi(N)$,  certainly  $N$ is non-abelian.  Set $N=N_1\times N_2\times \cdots \times N_r$ with $N_i \cong N_1$ a non-abelian simple group, for $i=2, \ldots, r$, and assume that $r >1$. 

Since $G=AN=BN$, both $A$ and $B$ act transitively by conjugation on the set $\Omega=\{N_1,  \ldots, N_r \}$. 

Suppose first that there exists some $p$-regular element of prime power order  $1 \neq x \in A \cup B$ such that $N_1^x=N_i$, for some $i >1$. By the hypotheses, there exists some $P \in \syl{p}{G}$ such that $P\leq \ce{G}{x}$. Moreover, $1 \neq P\cap N\in\syl{p}{N}$, and so $1\neq P\cap N_1\in\syl{p}{N_1}$. It follows that $P\cap N_1=(P\cap N_1)^x=P\cap N_1^x=P\cap N_i$, therefore  $1 \neq N_1 \cap N_i$, a contradiction. 

Hence, we may assume that any $p$-regular  element of prime power order in $A \cup B$ normalises $N_1$, and hence $N_i$ for $i=2, \ldots, r$, since $A$ and $B$ both act transitively on  $\Omega$. But this means that if $R:= \cap_{i=1}^{r} N_G(N_i)$, then $G=A_pR=B_pR$, for any $A_p \in \syl{p}{A}$ and $B_p \in \syl{p}{B}$. Therefore $G=PR$ for any $P \in \syl{p}{G}$, and so $P$ acts transitively on $\Omega$. But this contradicts the fact that $1 \neq  Z(P)  \cap N \leq C_N(P)=\ce{N_1}{P}\times \cdots \times \ce{N_r}{P}$, unless $r=1$. Therefore $N$ is a simple group as claimed. 
\end{proof}

In the next section the case when $N$ is a $p'$-group will be discarded. Hence, by Lemma \ref{3}, the minimal counterexample to our Main Theorem will be an almost simple group $G$, $N \unlhd G \leq \aut{N}$, with $p \in \pi(N)$ and $G=AB=AN=BN$. 

In Section 5 we will analyse such almost simple groups satisfying the hypotheses of our Main Theorem, and all possible cases for the simple group appearing as the socle of such a group will be ruled out. 

\subsection{Case $N$ is a $p'$-group.}

%In this section we will discard the case when $N$ is a $p'$-group, and so we deduce that our minimal counterexample is a group $G$ which is an almost simple group, $N \unlhd G \leq \aut{N}$, with $p \in \pi(N)$ and $G=AB=AN=BN$.

Assume from now on in this section that $N$ is a $p'$-group. Note that, in particular,  $G/N$ is $p$-decomposable, by Lemma~\ref{pclos}, and so $G$ is $p$-separable. 

\begin{lemma}\label{4} $G=\rad{p'}{G}\langle y\rangle=N\ce{G}{y}$, where $1 \neq y \in A$ and  $ \langle y \rangle \in \syl{p}{G}$. Further, $B$ is a $p'$-group. 
\end{lemma}
 \begin{proof}
Recall that $G/N=PN/N \times O_{p'}(G/N)$, for $P \in \syl{p}{G}$. Now since $N$ is a $p'$-group, it follows that $H:=\rad{p'}{G}\in\hall{p'}{G}$, so it is the unique  Hall $p'$-subgroup of $G$ and $H= (H \cap A)(H \cap B)$, by Lemma~\ref{1.3.2}. We may consider $P=(P\cap A)(P\cap B)$. For some $p$-element $y\in (P\cap A)\cup (P\cap B)$, set $H_y:=\rad{p'}{G}\langle y \rangle$. Note that $H_y= (H_y \cap A)(H_y \cap B)$. Now if $x\in (H_y\cap A)\cup (H_y \cap B)$ is a $p$-regular element  of prime power order, then there exists some $n\in N$ with $\langle y \rangle^n\leq  P^n\leq  \ce{G}{x}$, so $x\in\ce{H_y}{\langle y\rangle^n}$. As $\langle y\rangle^n$ is a Sylow $p$-subgroup of $H_y$ because $n\in N\leq  \rad{p'}{G}$, then $H_y$ satisfies the hypotheses (H2). If $\abs{H_y}<\abs{G}$, by minimality we obtain that $H_y$ has a normal Sylow $p$-subgroup, and so $[y, \rad{p'}{G}]=1$. If this holds for every $y\in (P\cap A)\cup (P\cap B)$, then $[P, \rad{p'}{G}]=1$, a contradiction. Hence we may suppose that, for instance, there exists $y\in P\cap A$ with $H_y=\rad{p'}{G}\langle y \rangle=G$. Further, since we are assuming that $\abs{A}+\abs{B}$ is minimal, then we deduce that $A=(\rad{p'}{G}\cap A)\langle y \rangle$ and $B=\rad{p'}{G}\cap B$.

Now, by coprime action and minimality,  $\rad{p'}{G}=[\rad{p'}{G}, y]\ce{\rad{p'}{G}}{y}\leq  N\ce{G}{y}.$ Thus $G=\rad{p'}{G}\langle y\rangle = N\ce{G}{y}$.
\end{proof}

\begin{lemma}\label{ncapa}
$N \cap A \neq 1$. 
\end{lemma}
\begin{proof}
Assume that  $N\cap A=1$.  By Lemma~\ref{1}, we know that $G/N=\langle y \rangle N/N \times \rad{p'}{G}/N$. Hence $[\langle y \rangle, \rad{p'}{G}\cap A]\leq  N\cap A=1$, so $\langle y\rangle$ is  a Sylow $p$-subgroup of $G$ which is normal  in $A$.
Now, since $B$ is a $p'$-group, we have that  for any $b \in B$ of prime power order, there exists $g \in G=AB$ such that $\langle y\rangle^g  \leq C_G(b)$. This implies that  $\langle y\rangle^{b_1}  \leq C_G(b)$ for some $b_1 \in B$, since $\langle y\rangle^a=\langle y\rangle$ for any $a \in A$. It follows that each element of prime power order of $B$ lies in $\underset{x\in B}{\cup}\ce{B}{\langle y\rangle}^x$ and so, by Lemma \ref{feinkantor}, we deduce $[B, \langle y\rangle]=1$, a contradiction which proves our claim.
% Now, since $B$ is a $p'$-group, 
%we have that  for any $b \in B$ there exists $g \in G=AB$ such that $\langle y\rangle^g  \leq C_G(b)$, and we may suppose that $g\in B$. It follows that $B=\underset{g\in B}{\cup}\ce{B}{\langle y\rangle}^g$ and so $[B, \langle y\rangle]=1$,  a contradiction which proves our claim.
\end{proof}

\begin{lemma}\label{5}  $N$ is a non-abelian group, so $N=N_1\times N_2\times \cdots \times N_r$, with $N_i \cong N_1$ a non-abelian simple group.
\end{lemma}
\begin{proof}
 Assume that $N$ is abelian. Therefore $\ce{N}{y}\leq  N$ is a normal subgroup of $G=N\ce{G}{y}$. Since $N$ is a minimal normal subgroup of $G$,  we deduce that either $N=\ce{N}{y}$ or  $\ce{N}{y}=1$. The first case yields to the contradiction $G=\ce{G}{y}$. So we may assume $\ce{N}{y}=1$. If we take $1 \neq x\in N\cap A$ of prime power order (which is a $p$-regular element) then, by our hypotheses, there exists some $n\in N$ such that $\langle y\rangle^n\leq  \ce{G}{x}$, and   so $x \in \ce{N}{\langle y\rangle^n}=(\ce{N}{\langle y\rangle}^n$, a contradiction. 
\end{proof}

%\begin{lemma}\label{xi}
%$\abs{N}\abs{\langle y\rangle}\abs{A\cap B}=\abs{\frac{G}{N}}\abs{N\cap A}\abs{N\cap B}.$
%\end{lemma}
%\begin{proof}
%Recall that $p\in\pi(AN)\smallsetminus\pi(BN)$ and $G=AN=BPN$ for any $P \in \syl{p}{G}$. Hence $G=AN=B\langle y\rangle N$. Now it is enough to make some calculations having in mind that 
%$\abs{B\cap \langle y\rangle N}=\abs{B\cap N}$ (recall that $G$ is $p$-separable).
%\end{proof}
\begin{lemma}\label{xi}
$\abs{N}\abs{\langle y\rangle}\abs{A\cap B}=\abs{\frac{G}{N}}\abs{N\cap A}\abs{N\cap B}.$
\end{lemma}
\begin{proof}
Recall that $p\in\pi(AN)\smallsetminus\pi(BN)$ and $G=AN=BPN$ for any $P \in \syl{p}{G}$. Hence $G=AN=B\langle y\rangle N$. Now it is enough to make some computations having in mind that 
$\abs{B\cap \langle y\rangle N}=\abs{B\cap N}$ (recall that $G$ is $p$-separable).
\end{proof}

\begin{lemma}\label{7} The Sylow $p$-subgroups of $G$ are cyclic of order $p$, i.e. $\langle y\rangle \cong C_p$.
\end{lemma}
 \begin{proof}
Take $x\in \langle y\rangle$ of order $p$, and set $H:=BN\langle x\rangle$; this is a subgroup of $G$ since $BN=\rad{p'}{G}\unlhd G$. Assume that $H=(H \cap A) B<G$.  Now if $h\in (H\cap A)\cup B$ is $p$-regular of prime power order, by the hypotheses there exists $n\in N$ with $\langle x\rangle ^n \leq \langle y\rangle^n \leq  \ce{G}{h}$, so $h\in \ce{H}{\langle x\rangle^n}$, where $\langle x\rangle^n\in\syl{p}{H}$. By minimality, $\langle x\rangle=\rad{p}{H}$, so (recall that $G=BPN$, by Lemma~\ref{1}, with $P=\langle y \rangle$) $1\neq \rad{p}{H}^G=\rad{p}{H}^{BNP}=\rad{p}{H}^P\leq  P$. This contradicts the fact $\rad{p}{G}=1$ and it proves the claim.
\end{proof}

\begin{lemma}\label{new}
	The subgroup $\langle y \rangle$ does not normalise  $N_i$, for each $i \in\{1, \ldots, r\}$. In particular, $r > 1$.
\end{lemma}
\begin{proof}
	Assume that $\langle y \rangle$ normalises some $N_i$ with $i \in\{1, \ldots, r\}$. Then $\langle y \rangle$ normalises  $N_i$ for each $i \in\{1, \ldots, r\}$. We can view $\langle y \rangle$ as a subgroup of $\op{Aut}(N_i)$, because $\ce{\langle y\rangle}{N_i}=1$ (recall that $y$ has order $p$). By Lemma~\ref{aut}, there exists a prime $s\in\pi(N_i)\smallsetminus\pi(\op{Out}(N_i))$ such that $(s, |\ce{N_i}{y}|)=1$. Therefore $s$ cannot divide $|\ce{N}{y}|$ as $\ce{N}{y}=\ce{N_1}{y}\times \cdots \times \ce{N_r}{y}$. 

Since each element of prime power order in $(N \cap A) \cup (N \cap B)$ centralises some Sylow $p$-subgroup (because our hypotheses), we deduce that $\pi(N\cap A)\cup\pi(N\cap B)\subseteq \pi(\ce{N}{y})$. Thus, this last property and Lemma~\ref{xi} yield $s\in\pi(G/N)$. 

Note that $G/N\lessapprox \op{Out}(N)$ and $\op{Out}(N)\cong \op{Out}(N_1) \text{ wr } \Sigma_{r}$, , the natural wreath product of 
$\op{Out}(N_1)$ with $\Sigma_r$. As $s$ does not divide $\abs{\op{Out}(N_i)}$ for any $i$, it follows that $s\in\pi(\Sigma_r)$. 
Using Lemma~\ref{xi}, we obtain that $\abs{N}_s$ divides $\abs{G/N}_s$, and so it divides $\abs{\Sigma_r}_s$. Set $\abs{N_1}_s:=s^d$. Then $\abs{N}_s=s^{dr}$ divides $s^{\frac{r-1}{s-1}}$, by Lemma \ref{SylowSym}, so $dr\leq \frac{r-1}{s-1}$ and necessarily $d=0$, which contradicts the fact that $s$ divides $\abs{N_i}$. 
\end{proof}

\begin{lemma}
 Set $C:=\ce{\rad{p'}{G}\cap A}{y}=\ce{\rad{p'}{A}}{y}$ and $A_0:=\langle y \rangle \times C$. Then $A=(N\cap A)A_{0}$ and $G=AN=A_{0}N$. 
\end{lemma}
\begin{proof}
By the minimal choice of $G$, we deduce that $G/N\cong A/A\cap N$ is $p$-decomposable. Hence $[\rad{p'}{G}\cap A, \langle y\rangle ]\leq  N\cap A$. Thus, by coprime action, $\rad{p'}{G}\cap A=[\rad{p'}{G}\cap A, \langle y\rangle]\ce{\rad{p'}{G}\cap A}{y}\leq  (N\cap A)\ce{\rad{p'}{G}\cap A}{y}$. Hence $A=(N\cap A)\ce{\rad{p'}{G}\cap A}{y}\langle y\rangle=(N \cap A)A_0$, and the assertion follows. \end{proof}
\medskip

Recall that $N=N_1\times N_2\times \cdots \times N_r$ with $N_i \cong N_1$ a non-abelian simple group, and set $\Omega:=\{N_1, N_2, \ldots, N_r\}$. By Lemma \ref{new},  $r > 1$. As $G=A_0N$,  $A_0$ acts transitively on $\Omega$. We adapt here some arguments used in \cite{KMP3} and we claim some facts about this action:

\begin{itemize}

\item[(i)] The orbits of $B$ on $\Omega$ are the same as those of $C$. 

This is clear because $\rad{p'}{G}=\rad{p'}{A}N=CN=BN$.

\item[(ii)] Let $\Delta$ be an orbit of $C$ on $\Omega$ of minimal lenght. If $c\in C$, then  $\Delta^{yc}=\Delta^{cy}=\Delta^y$, so $\Delta^y$  and $\Delta\cap \Delta^y$ are also orbits of $C$. Therefore, by the choice of $\Delta$,  either $\Delta=\Delta^y$ (and hence $\Delta=\Delta^{y^{i}}$ for $i\in\{1,\ldots, p\}$),  or $\Delta\cap \Delta^y=\emptyset$ (and hence $\Delta^{y^{i}}\cap \Delta^{y^{j}}=\emptyset$ for $i\neq j$, $i, j \in \{1,\ldots, p\}$). It follows that there is a partition of $\Omega$ of the form $$\Omega=\Delta_1 \cup \Delta_2 \cup \cdots \cup \Delta_k,$$ where $\Delta_i:=\Delta^{y^{i-1}}$ for $i\in\{1,\ldots, k\}$, and $k \in \{1, p\}$. Note that all $\Delta_i$ have the same length, say $m$, and  $\{\Delta_1, \ldots,  \Delta_k\}$ are all the $C$-orbits (and $B$-orbits) on $\Omega$. Note also that $m$ is a $p'$-number, since $C$ is a  $p'$-subgroup.

Moreover, $\langle y \rangle$ acts transitively on $\{\Delta_1, \Delta_2, \ldots, \Delta_k\}$. 

\item[(iii)] The length of an orbit $\nabla$ of $\langle y \rangle$ on $\Omega$ is $k=p$, and there are $m$ orbits $\nabla_1:=\nabla$, $\nabla_2\dots, \nabla_m$. Hence there is a partition $$\Omega=\nabla_1\cup \cdots \cup \nabla_m$$ 
and both $\rad{p'}{A}$ and  $B$ act transitively on the set  $\{\nabla_1, \nabla_2, \dots, \nabla_m\}$. In particular, for each $1\leq i\leq m$, there exists $a_i\in\rad{p'}{A}$ such that $\nabla_1^{a_i}=\nabla_i$; $a_1=1$.

 Since the lenght of an orbit of $\langle y \rangle$ on $\Omega$ divides $p$ and  $\langle y \rangle$ does not normalise any $N_i$, by Lemma \ref{new}, the first assertion follows. 
Now, the fact that $G=\langle y \rangle \rad{p'}{A}N=\langle y\rangle BN$ gives the last assertion.

\item[(iv)] It follows from (ii) and (iii) that $r=pm$, with $1=(m,p)$.

% where $k=p$ is the number of $C$-orbits and $m\geq 1$ is the length of a $C$-orbit (or, equivalently, $p$ is the length of an orbit of $\langle y \rangle$ on $\Omega$ and $\langle y \rangle$ has $m$ orbits on $\Omega$). 

\item[(v)] Without loss of generality, we may consider $\Delta=\{N_1,\ldots, N_m\}$, and we set $M_{\Delta}:=N_1\times \cdots \times N_m$. Then $M_{\Delta}$ is a minimal normal subgroup of $NC$. 

Moreover, if $1\neq R\leq  N$ and $R\unlhd NC$, then there exist $\{x_1, \ldots, x_d\}\subseteq \langle y \rangle$ such that $R=M_{\Delta}^{x_1} \times \cdots \times M_{\Delta}^{x_d}$.

\item[(vi)] Since $r > 1$, then $m>1$.

Recall that $r >1$, by Lemma \ref{new}. If $m=1$, then $\langle y \rangle$ has only one orbit on $\Omega$, i.e. $\langle y \rangle$ acts transitively on $\Omega=\{N_1, \ldots, N_k\}$.   Suppose, for instance, that $N_i=N_1^y$, for $i > 1$. If there exists a non-trivial element $x\in\ce{N_1}{y}$, then $x=x^y\in N_1\cap N_1^y=N_1 \cap N_i$, a contradiction. Hence $\ce{N_1}{y}=1$, and  it follows that $\ce{N}{y}=1$. But since $N\cap A\neq 1$, by Lemma \ref{ncapa}, we can choose an element of prime powe order $1 \neq x \in N \cap A$  such that $x \in C_N(y)$ because the hypotheses on the indices (recall that $N$ is a $p'$-group), which gives a contradiction.

\item[(vii)]  Since $k=p>1$, then $A\cap M_{\Delta}=1=\ce{\rad{p'}{A}}{M_{\Delta}}$.

Let $x\in A\cap M_{\Delta}$ of prime power order, which is a $p$-regular element. Then by hypotheses there exists $n\in N$ with $x\in\ce{G}{\langle y\rangle^n}$. Hence $x^{y^n}=x\in M_{\Delta}\cap M_{\Delta}^{y^n}=M_{\Delta}\cap M_{\Delta}^{y}=1$. We deduce that $A\cap M_{\Delta}=1$.

Let $x\in \ce{\rad{p'}{A}}{M_{\Delta}}$ of prime power order. Then there exists $n\in N$ such that $[\langle y \rangle^n, x]=1$. Therefore $[x, M_{\Delta}]=1=[x^{(y^j)^n}, M_{\Delta}^{(y^j)^n}]=[x, M_{\Delta}^{y^j}]$, for every $j\in\{1, \ldots, p-1\}$. So $x\in\ce{G}{N}=1$. Hence $\ce{\rad{p'}{A}}{M_{\Delta}}=1$.

%\item[(vii)] We consider now the $\langle y \rangle$-orbits of $\Omega$.  Let $\nabla$ be a $\langle y \rangle$-orbit of $\Omega$. There is a partition $\Omega=\nabla_1\cup \cdots \cup \nabla_m$ with $\nabla_1:=\nabla$, and $\rad{p'}{A}$ (and also $B$) acts transitively on $\{\nabla_1,\ldots, \nabla_m\}$ as $G=\langle y \rangle \rad{p'}{A}N=\langle y\rangle BN$. 

%So for every $1\leq i\leq m$, there exists $a_i\in\rad{p'}{A}$ such that $\nabla_1^{a_i}=\nabla_i$; $a_1=1$. 
%
%Without loss of generality, let $\nabla=\{N_1, \ldots, N_k\}$ (with $k \in \{1, p\}$), and set $M_{\nabla}:=N_1\times \cdots\times N_k$. Then $M_{\nabla}$ is a minimal normal subgroup of $N\langle y \rangle$. Moreover, if $1\neq R\leq  N$ with $R\unlhd N\langle y\rangle$, then there exist $\{d_1,\ldots, d_t\}\subseteq\{a_1,\ldots, a_m\}$ such  that $R=M_{\nabla}^{d_1}\times \cdots \times M_{\nabla}^{d_t}$.

\item[(viii)] Without loss of generality, let $\nabla =\{N_1, \ldots, N_p\}$.  Set $M_{\nabla}:=N_1\times \cdots\times N_p$. Then $M_{\nabla}$ is a minimal normal subgroup of $N\langle y \rangle$, and if $1\neq R\leq  N$ with $R\unlhd N\langle y\rangle$, then there exist $\{d_1,\ldots, d_t\}\subseteq\{a_1,\ldots, a_m\} \subseteq   \rad{p'}{A}$ such  that $R=M_{\nabla}^{d_1}\times \cdots \times M_{\nabla}^{d_t}$.

Moreover, if we set $F_1:=N_2\times \cdots \times N_p$,  $F_i:=F_1^{a_i}$ for each $2\leq i\leq m$, and $F_{\nabla}:=F_1\times \cdots \times F_m$, then  $F_{\nabla}\cap \rad{p'}{A}=1=F_{\nabla}\cap B$.

 The first assertion is clear. If $x\in F_{\nabla}\cap \rad{p'}{A}$ is of prime power order, then there exists $n\in N$ such that $\langle y \rangle^n$ centralises $x$, so for every $1\leq j \leq p$ we get $x=x^{(y^j)^n}\in F_{\nabla}\cap F_{\nabla}^{y^j}\leq  E_{\nabla}:=\cap_{g\in\langle y \rangle} F_{\nabla}^g$. It follows that $F_{\nabla}\cap \rad{p'}{A}\leq  E_{\nabla}$. Note that $E_{\nabla}\leq  N$ and it is normal in $N\langle y \rangle$, hence we deduce from the above that  $E_{\nabla}=1$, and so $F_{\nabla}\cap \rad{p'}{A}=1$. Analogously, $F_{\nabla}\cap B=1$.
\end{itemize}
\medskip
Now, we will use the above facts on the actions of $C$ (and so $B$) and $\langle y \rangle$ on the set $\Omega$ to see that the minimal normal subgroup $N$ in our minimal counterexample cannot be a $p'$-group. 

The proof of the next Lemma follows similar arguments as those in \cite[Lemma 11]{KMP3}, using (i)-(viii) above, with suitable changes. However, we include an outline of the proof for the convenience of the reader.

\begin{lemma}\label{ix} Let $s\neq p$ be a prime, and assume that $\abs{N_1}_s=s^n$ and $\abs{\op{Out}(N_1)}_s=s^{\delta}$. Then $n(p-2)\leq \delta + \frac{m-1}{m(s-1)}$, where $r=pm$. In particular, $n(p-2)< \delta+1$.
\end{lemma}
\begin{proof}
Recall that   a $\langle y \rangle$-orbit $\nabla$ on $\Omega$ has length $k=p > 1$. 
Let $A_s\in\syl{s}{A}$ and $B_s\in\syl{s}{B}$. Note that $A_s\leq  \rad{p'}{A}$. By (viii) above, $F_{\nabla}\unlhd N$ and $F_{\nabla}\cap N\cap A_s\leq  F_{\nabla}\cap \rad{p'}{A}=1$. So it follows that $\abs{A_s\cap N}\leq \abs{N:F_{\nabla}}_s=\abs{N_1}_s^m=s^{nm}$. Analogously $\abs{B_s\cap N}\leq s^{nm}$. 

 Set $M:=M_{\Delta}=N_1 \times \cdots \times N_m$. From (v) and (vii) above we have that  $M\unlhd N\rad{p'}{A}=NB$ and $M\cap \rad{p'}{A}=1=\ce{\rad{p'}{A}}{M}$. Hence $\rad{p'}{A}\cong \rad{p'}{A}\ce{G}{M}/{\ce{G}{M}} \lessapprox \op{Aut}(M)$.
Moreover $\op{Aut}(M) \cong [\op{Aut}(N_1) \times \cdots \times \op{Aut}(N_m)]\Sigma_m \cong \, \op{Aut}(N_1)\op{wr} \, \Sigma_m$, the natural wreath product of 
$\op{Aut}(N_1)$ with $\Sigma_m$. 
Now applying Lemma~\ref{SylowSym} we deduce that $\abs{A_s}$ divides $\abs{\op{Aut}(M)}$, and so $s^{(\delta+n)m}\cdot s^{\frac{m-1}{s-1}}$.

On the other hand, if $\abs{G/N}_s=s^{\gamma}$, then $\abs{G}_s=\abs{G/N}_s \abs{N}_s=s^{\gamma+nr}$. Further, $\abs{B_s}=\abs{G/N}_s \abs{B_s\cap N}$ divides $s^{\gamma+nm}$.
Since $\abs{G}_s$ divides $\abs{A_s} \abs{B_s}$, so $s^{\gamma+nr}$ divides $s^{\frac{m-1}{s-1}}\cdot s^{\gamma+nm} s^{(\delta + n)m}$. This fact, after some straightforward computations,  leads to the desired conclusion, having in mind that $r=pm$. 
\end{proof}

\begin{lemma}  $N$ is not a $p'$-group. 
\end{lemma}
\begin{proof}
We take a prime $s\in \pi(N_1)\smallsetminus \pi(\op{Out}(N_1))$ (such prime always exists, see for instance \cite[Lemma 5]{KMP3}). Note that $s \neq p$, since $N$ is a $p'$-group. Applying the previous Lemma for such prime we obtain that $n(p-2)<\delta +1$, but $\delta=0$ so necessarily $p=2$. This cannot happen, as it would imply that $N$ is a $2'$-group, so soluble, which is a contradiction by Lemma~\ref{5}.
\end{proof}

\section{The almost simple case}

Let $N$ be a non-abelian simple group with $p \in \pi(N)$, and let $N \unlhd G \leq \aut{N}$ such that $G=AB=AN=BN$. Assume that $G$ satisfies the hypotheses of our main theorem, i.e. 
 $p$ does not divide $i_G(x) $ for every $p$-regular element  of prime power order $x \in A \cup B$.

 We will  carry out a case-by-case analysis of the simple group $N$ occuring as the socle of $G$ to prove that there is no a counterexample to our Main Theorem. Our strategy will apply the following lemma and the results in Section 3. 
\begin{lemma}\label{cent}
For any $P \in \syl{p}{G}$  % \setminus \{p\}
$$\pi(G)=\pi(C_G(P)) =\pi(G/N) \cup \pi(C_N(P)).$$
In particular, $p \in \ze{\Gamma(G)}$. 
Moreover,  if $r \in \pi(G) \setminus \pi(G/N)$, then $r$ is adjacent to $p$ in $\Gamma(N)$. 
\end{lemma}
\begin{proof}
 By our hypotheses,  for any prime $r \in \pi(G) \setminus \{p\}$ there exists an element $x \in A \cup B$ of order $r$ such that $x \in C_G(P)$, for  some $P \in \syl{p}{G}$.  Hence the first equality follows. Now, observe that $\pi(G)=\pi(G/N) \cup \pi(N)$ and, since $G=AN=BN=AB$, after some computations we also obtain
$$\abs{N}\abs{A\cap B}=\abs{\frac{G}{N}}\abs{N\cap A}\abs{N\cap B}.$$
But again our hypotheses lead to $\pi((N \cap A) \cup (N \cap B))\setminus \{p\} \subseteq \pi(C_{N}(P))$. Also, since $p \in \pi(N)$, $1 \neq Z(P) \cap N \leq C_{N}(P)$. Hence the second equality also holds. 

It is clear then that $p \in \ze{\Gamma(G)}$. Assume now that  $r \in \pi(G) \setminus \pi(G/N)$, and so $r \in \pi(N)$. By the second equality, we deduce that $r \in  \pi(C_N(P))$ and since $p \in \pi(N)$ the last assertion follows.
\end{proof}

\begin{lemma}\label{notan}
$N$ is not an alternating group $A_n$.
\end{lemma}
\begin{proof}
Let $N=A_n$ and assume first $n\neq 6$, so $G=N=A_{n}$ or $G=\Sigma_n$.   Because our hypotheses, we may assume that $\Gamma(G)$ is connected. As in Lemma \ref{angraph}, let $k\geq 2$ be  the largest positive integer such that $\{n, n-1, \ldots, n-k+1\}$ are consecutive composite numbers,  $r:=n-k$ the largest prime divisor of $n!$, and  $t$ the largest prime with $t \leq k$. Since $p \in \ze{\Gamma(G)}$,  then $p\leq t$ by  Lemma \ref{angraph}, and so $r>\frac{n}{2}>k\geq t\geq p$. 

We claim that $r\notin\pi(\ce{G}{P})$, for $P\in\syl{p}{G}$. Let suppose first that $p\neq 2$. Assume that there exists an element $x\in G$ of order $r$ such that $P\leq \ce{G}{x}$. Since $\ce{G}{x}$ is isomorphic to a subgroup of $C_r\times \Sigma_{n-r}$ and $p\neq r$, then $\abs{P}$ divides $\abs{\Sigma_{n-r}}$, and so $\abs{\Sigma_n:\Sigma_{n-r}}= n(n-1)\cdots (n-r+1)$ should be a $p'$-number. But this is a contradiction, since $p<r$.

If $p=k=2$, then, by Lemma \ref{angraph}, it should be $G=\Sigma_n$, so the above reasonings work as well. Finally, if $p=2$ and $k\geq 3$, then $r\geq 5$ and we can argue as above to get a contradiction since $\abs{\Sigma_n:\Sigma_{n-r}}= n(n-1)\cdots (n-r+1)$ is divisible by $4$.

If $n=6$, by Lemma~\ref{a6}, the only case to be considered is $G=\op{Aut}(N)$, and since  $\ze{\Gamma(G)}=\{2\}$ it should be $p=2$. But a Sylow $2$-subgroup of $G$ is self-centralising, so we get a contradiction. 
\end{proof}

\begin{lemma}\label{notsp}
$N$ is not an sporadic group.
\end{lemma}
\begin{proof}
Assume that $N$ is an sporadic group. Since $p \in \ze{\Gamma(G)}$, we may assume, by Lemma \ref{sporgraph}, that either $N= J_2$ or $N=McL$, $G=\aut{N}$, and $p=2$. Now Lemma \ref{cent} implies that $2$ is adjacent in $N$ to any prime $r\neq 2$, but this is a contradiction since $N$ has a self-centralising Sylow $s$-subgroup (take $s=7$ for $N=J_2$ and  $s=11$ for $N=McL$; see \cite{Atl}). 
\end{proof}

\begin{lemma}
$N$ is not a simple group of Lie type.
\end{lemma}
\begin{proof}
If $N$ is a simple group of Lie type of characteristic $t$, first notice that the prime $p$ such that $\pi(G)=\pi(C_G(P))$  should be different from $t$, because it is well known that a Sylow $t$-subgroup is self-centralising in $G$. 
Moreover, since $\abs{N}\abs{A\cap B}=\abs{\frac{G}{N}}\abs{N\cap A}\abs{N\cap B}$ and $|N|_t > |\out{N}|_t$,  we get that $t \in \pi((N \cap A) \cup (N \cap B)) \subseteq \pi(C_{N}(P))$. This means that $t$ should be adjacent to $p$ in $\Gamma(N)$.

Now, we derive from Lemmas \ref{class} and \ref{excep} that, apart from some exceptional cases that we consider below, either there exist a Sylow $s$-subgroup of $N$ of order $s \not \in \pi(\out{N})$ which is self-centralising in $N$, or there exist two primes $r, s \in \pi (N)\setminus \pi(G/N)$, and two maximal tori $T_1$ and $T_2$ of $N$ such that $r \in \pi(T_1)$, $s\in \pi(T_2)$, and $(|T_1|, |T_2|)=1$. But from  Lemma \ref{cent}, $p$ is a prime which is adjacent both to $r$ and $s$ in $\Gamma(N)$, and therefore $p \in \pi(T_1) \cap \pi(T_2)$, which gives a contradiction. 

For $N=L_2(q)$, $q=t^e$,  the fact that $C_N(x)$ is a $t$-group for any $t$-element $x \in N$, implies that $t$ is not adjacent in  $\Gamma(N)$ to any other prime in $\pi(N)$, a contradiction.

If $N=L_3(q)$ or $N=U_3(q)$, $q=t^e$, the assertion in Lemma \ref{class} on the corresponding maximal torus $T$ in each case guarantees that $p \in \pi(T)$ is a primitive prime divisor of $q^3-1$ (respectively $q^6-1$) and $p$ is not adjacent to the prime $t$ in $\Gamma(N)$. In fact, $p$ is not adjacent in $\Gamma(N)$ to any prime $s \not \in \pi(T)$, which gives a contradiction. 
\end{proof}

The Main Theorem is proved.\\

\noindent \textbf{Acknowledgements.}  Research supported by  Proyecto PGC2018-096872-B-I00 from the Ministerio de Ciencia, Innovaci\'on y Universidades, Spain, and FEDER.  The second author is also supported  by Project VIP-008 of Yaroslavl P. Demidov State University and the third author by Proyecto PROMETEO/2017/057 from the Generalitat Valenciana, Spain.

\end{document}